\documentclass[a4paper,12pt]{article}
\usepackage{amsmath}
\usepackage{amsfonts}
\usepackage{amsthm}
\usepackage{amssymb}
\usepackage{bbm}
\usepackage{lipsum}

\newtheorem{theorem}{Theorem}[section]
\newtheorem{lemma}[theorem]{Lemma}

\newtheorem{corollary}[theorem]{Corollary}

\newcommand{\R}{\mathbb R}

\newcommand\blfootnote[1]{%
	\begingroup
	\renewcommand\thefootnote{}\footnote{#1}%
	\addtocounter{footnote}{-1}%
	\endgroup
}

\DeclareMathOperator{\mgr}{mgr}

\title{The Maximal Generalised Roundness of Finite Metric Spaces}
\author{Gavin Robertson}

\begin{document}

\maketitle

\begin{abstract}
We provide two simplifications of S{\'a}nchez's formula for the maximum generalised roundness of a finite metric space.
\end{abstract}

\section{Introduction}

\blfootnote{MSC Primary 52C99; Secondary 43A35;}

The equivalent concepts of $p$-negative type and generalised roundness arose in embedding theory in the work of Schoenberg \cite{Schoenberg 2} and Enflo \cite{Enflo 1}. Let $(X,d)$ be a metric space and suppose that $p \ge 0$. Then
\begin{enumerate}
 \item $(X,d)$ is of $p$-negative type if
  \[ \sum_{i,j=1}^{n} d(x_i,x_j)^p \xi_i\xi_j \le 0 \]
for all finite collections $x_1,\dots,x_n \in X$ and $\xi_1,\dots,\xi_n \in \R$ with $\xi_1+\dots+\xi_n = 0$.
 \item $(X,d)$ has generalised roundness $p$ if
 	\[ \sum_{i,j=1}^{n}d(x_{i},x_{j})^{p}+d(y_{i},y_{j})^{p}\leq2\sum_{i,j=1}^{n}d(x_{i},y_{j})^{p}
 	\]
for all finite collections $x_{1},\dots,x_{n},y_{1},\dots,y_{n}\in X$.
\end{enumerate}
It was shown by Lennard, Tonge and Weston that the set of $p$ for which a metric space is of $p$-negative type is the same as the set for which it has generalised roundness $p$ (\cite{Lennard 1}, Theorem 2.4). In both cases this set is of the form $[0,\wp]$ or $[0,\infty)$. It therefore makes sense to define the maximal generalised roundness (or supremal $p$-negative type) of $(X,d)$ which we will denote $\mgr(X,d)$, or $\mgr(X)$ if the metric is understood.

The most classical result in this area is Schoenberg's embedding theorem \cite{Schoenberg 1} which states that a metric space embeds isometrically in a Hilbert space if and only if $\mgr(X,d) \ge 2$. But links have been found between this concept and many apparently unrelated properties of certain types of spaces. For example, any simple connected graph $G$ carries the path metric. Hjorth \textit{et al.} (\cite{Hjorth 1}, Corollary 7.2) showed that if $\mgr(G) < 1$ then $G$ must contain a cycle.

In general however, calculation of $\mgr(X,d)$ has been challenging, even for finite metric spaces. A major step forward was provided by S{\'a}nchez \cite{Sanchez 1}, who provided a formula for the maximal generalised roundness of a finite metric space in terms of its distance matrix.

In what follows, we shall always assume that $X=\{x_{0},x_{1},\dots,x_{n}\}$ is a finite metric space, with distance matrix $D = (d(x_i,x_j))_{i,j=0}^n$. For $p>0$, let $D_p$ denote the `$p$-distance matrix', $(d(x_i,x_j)^p)_{i,j=0}^n$. Also, let $\mathbbm{1}$ denote the column vector in $\R^{n+1}$ all of whose coordinates are $1$, and let $\langle\cdot,\cdot\rangle$ denote the standard inner product on $\R^{n+1}$. S{\'a}nchez (\cite{Sanchez 1}, Corollary 2.4) showed that if $\mgr(X)<\infty$ then
	\begin{equation}\label{Sanchez's Formula}
	\mgr(X) = \min\{p>0 \,:\, \text{$\det(D_p) = 0$ or $\langle D_{p}^{-1}\mathbbm{1},\mathbbm{1}\rangle$}=0\}.
	\end{equation}
At least for small finite metric spaces this enables one to numerically calculate $\mgr(X)$. For explicit computations of $\mgr(X)$ for certain finite metric spaces using this formula, see for example Theorem 3.1 of \cite{Sanchez 1} or Proposition 4.4 of \cite{Faver 1}.

The aim  of this note is to show that the above expression can be replaced by simpler ones which only involve the calculation of a single determinant. To state our results we require some more notation. For $p>0$ we define the Cayley-Menger matrix $M_{p}$ by
	\[ M_{p}=\begin{pmatrix} 0 & \mathbbm{1}^{T} \\ \mathbbm{1} & D_{p} \end{pmatrix}.
	\]
Note that the determinant of $M_{p}$ is nothing but the classical Cayley-Menger determinant of the metric space $(X,d^{p/2})$. We also define the Gramian $G_{p}=(g_{ij})_{i,j=1}^{n}$ to be the matrix with entries
	\[ g_{ij}=\frac{1}{2}(d(x_{i},x_{0})^{p}+d(x_{j},x_{0})^{p}-d(x_{i},x_{j})^{p}).
	\]
The most important result pertaining to the Gramian is the following refinement of Schoenberg's embedding theorem: the metric space $(X,d)$ embeds isometrically into $(\R^{m},\|\cdot\|_{2})$ if and only if $G_{2}$ is positive semidefinite with rank at most $m$. Also, we have the following fundamental relationship between the Cayley-Menger matrix and the Gramian. If $p>0$ then
	\begin{equation}\label{Cayley-Menger Formula}
	\det(G_{p})=\frac{(-1)^{n+1}}{2^{n}}\det(M_{p}).
	\end{equation}

For more on Cayley-Menger determinants and Gramians, including the above formula, see \cite{Maehara 1}.

We may now state our main result.

\begin{theorem}\label{Main Theorem}
	Let $(X,d)$ be a finite metric space with $\mgr(X)<\infty$. Then
		\[ \mgr(X)=\min\{p>0:\det(M_{p})=0\}=\min\{p>0:\det(G_{p})=0\}.
		\]
\end{theorem}

The difference between the expressions in the above theorem and that given in (\ref{Sanchez's Formula}) may seem small at first glance. However, it turns out that in many ways Theorem \ref{Main Theorem} fits into the theory of generalised roundness far more naturally than (\ref{Sanchez's Formula}). For a start, the formula in (\ref{Sanchez's Formula}) reduces the computation of $\mgr(X)$ to an analysis of $\det(D_{p})$ and $\langle D_{p}^{-1}\mathbbm{1},\mathbbm{1}\rangle$ as functions of $p$. The analysis of $\langle D_{p}^{-1}\mathbbm{1},\mathbbm{1}\rangle$ is almost always problematic since this expression is not necessarily even defined for all $p>0$. Also, even where it is defined it more or less requires the computation of $D_{p}^{-1}$, which is in most cases far from easy. The formulae in Theorem \ref{Main Theorem} do not have these problems, as they instead only require an analysis of the more well behaved quantities $\det(M_{p})$ and $\det(G_{p})$. That is to say, they only require an analysis of quantities which are defined for all $p>0$ and do not make use of the problematic $D_{p}^{-1}$.

To actually demonstrate the usefulness of Theorem \ref{Main Theorem} we remark that it can be used to give a much simpler proof of a theorem of Murugan (\cite{Murugan 1}, Theorem 4.3), concerning the maximal generalised roundness of subsets of the Hamming cube. This simpler proof is also based on the following identity from \cite{Doust 1}.

\begin{theorem}\label{New Determinant Formula}
	Let $X=\{x_{0},x_{1},\dots,x_{m}\}\subseteq(H_{n},\ell^{1})$ be a subset of the $n$-dimensional Hamming cube with $x_{0}=0$, and distance matrix $D_{1}$. Then if $B$ is the $m\times n$ matrix whose $i$-th row is $x_{i}$ for $1\leq i\leq m$, we have that
	\[ \det(M_{1})=(-1)^{m-1}2^{m}\det(BB^{T}).
	\]
\end{theorem}

\begin{corollary}\label{Murugan's Theorem}\textit{(Murugan's Theorem).}
	A subset $X=\{x_{0},x_{1},\dots,x_{m}\}\subseteq (H_{n},\ell^{1})$ with $x_{0}=0$ has $\mgr(X)>1$ if and only if it is affinely independent. 
\end{corollary}

For the details of this simpler proof of Corollary \ref{Murugan's Theorem}, see section 2.

Finally, we remark that Theorem \ref{Main Theorem} holds equally well for semi-metric spaces (metric spaces where we do not necessarily require that the triangle inequality holds).

\section{Technical Proofs}

We start by proving a formula concerning the product $\det(D_{p})\langle D_{p}^{-1}\mathbbm{1},\mathbbm{1}\rangle$.
	
\begin{lemma}\label{New Determinant}
	If $D_{p}$ is invertible, then
	\[ \det(M_{p})=-\det(D_{p})\langle D_{p}^{-1}\mathbbm{1},\mathbbm{1}\rangle.
	\]
\end{lemma}
\begin{proof}
	Let $I$ denote the identity matrix of size $n+1$ and $0_{n+1}$ the vector in $\R^{n+1}$ all of whose coordinates are $0$. It is a simple matter to verify the matrix identity (in block form)
		\[ \begin{pmatrix} 0 & \mathbbm{1}^{T} \\ \mathbbm{1} & D_{p} \end{pmatrix}=\begin{pmatrix} 1 & \mathbbm{1}^{T} \\ 0_{n+1} & D_{p} \end{pmatrix}\begin{pmatrix} -\mathbbm{1}^{T}D_{p}^{-1}\mathbbm{1} & 0_{n+1}^{T} \\ D_{p}^{-1}\mathbbm{1} & I \end{pmatrix}.
		\]
	The result now follows by taking the determinant of both sides and noticing that
		\[ -\mathbbm{1}^{T}D_{p}^{-1}\mathbbm{1}=-\langle D_{p}^{-1}\mathbbm{1},\mathbbm{1}\rangle.
		\]
\end{proof}

	We also require the following result of S\'{a}nchez (\cite{Sanchez 1}, Corollary 2.5).

\begin{theorem}\label{Dichotomy}
		Let $(X,d)$ be a finite metric space with $\wp=\mgr(X)<\infty$. Also, let $\Pi_{0}$ be the plane in $\R^{n+1}$ with equation $\xi_{0}+\xi_{1}+\dots+\xi_{n}=0$. Then we have the following dichotomy.
		\begin{enumerate}
			\item Either $\det(D_{\wp})=0$, in which case for any $\xi\in\ker(D_{\wp})$ we have $\xi\in\Pi_{0}$, or else
			\item $\det(D_{\wp})\neq 0$ and $\langle D_{\wp}^{-1}\mathbbm{1},\mathbbm{1}\rangle=0$.
		\end{enumerate}
	\end{theorem}

We may now prove Theorem \ref{Main Theorem}.
	
\begin{proof}(Of Theorem \ref{Main Theorem})
		We will prove the formula concerning $\det(M_{p})$. The second part of the formula that pertains to $\det(G_{p})$ then immediately follows from this and (\ref{Cayley-Menger Formula}). So, first suppose that $0<p<\wp=\mgr(X)$. Then by (\ref{Sanchez's Formula}) we know that
		\[ \det(D_{p})\neq0\mbox{ and }\langle D_{p}^{-1}\mathbbm{1},\mathbbm{1}\rangle\neq 0.
		\]
		But then Lemma \ref{New Determinant} clearly implies that
		\[ \det(M_{p})\neq0.
		\]
		Now all that we need to show is that the above determinant is zero at $p=\wp$. To do so, we just consider both of the cases in Theorem \ref{Dichotomy}. The second case is easier. Indeed, suppose it is the case that
		\[ \det(D_{\wp})\neq 0\mbox{ and }\langle D_{\wp}^{-1}\mathbbm{1},\mathbbm{1}\rangle=0.
		\]
		Then Lemma \ref{New Determinant} again clearly implies that $\det(M_{\wp})=0.$ Now suppose instead that $\det(D_{\wp})=0.$ Then $D_{\wp}$ is not invertible and so there exists a nonzero $\xi=(\xi_{0},\xi_{1},\dots,\xi_{n})\in\ker(D_{\wp})$. By Theorem \ref{Dichotomy} we have that $\xi\in\Pi_{0}$. We claim that $\eta=(0,\xi)=(0,\xi_{0},\xi_{1},\dots,\xi_{n})\in\R^{n+2}$ is an element of the kernel of $M_{\wp}$. Indeed, since $\xi\in\Pi_{0}$ we have that
		\[ M_{\wp}\eta=\begin{pmatrix} 0 & \mathbbm{1}^{T} \\ \mathbbm{1} & D_{\wp} \end{pmatrix}\begin{pmatrix} 0 \\ \xi \end{pmatrix}=\begin{pmatrix} \sum_{i=0}^{n}\xi_{i} \\ D_{\wp}\xi \end{pmatrix}=\begin{pmatrix} 0 \\ 0_{n+1} \end{pmatrix}.
		\]
		But since $\xi$ is nonzero it follows that $\eta$ is also nonzero. Hence we conclude that
		\[ \det(M_{\wp})=0
		\]
		since we have found a nonzero element in the kernel of $M_{\wp}$.
\end{proof}

Next we give the proof of Corollary \ref{Murugan's Theorem}.

\begin{proof}(Of Corollary \ref{Murugan's Theorem})
	Since all subsets of the Hamming cube have $\mgr(X)\geq 1$ (\cite{Lennard 1}, Corollary 2.6) Theorem \ref{Main Theorem} implies that
	\[ \det(M_{p})\neq0
	\]
	for all $0<p<1$. Keeping the notation from Theorem \ref{New Determinant Formula}, it is a basic fact of linear algebra that $x_{1},\dots,x_{m}$ are linearly independent if and only if $\det(BB^{T})\neq0$ (see \cite{Gantmacher 1} for a classical treatment of such results). Thus by Theorem \ref{New Determinant Formula} we have that $\det(M_{1})\neq0$ if and only if $x_{1},\dots,x_{m}$ are linearly independent. Hence Theorem \ref{Main Theorem} implies that $\mgr(X)>1$ if and only if $x_{1},\dots,x_{m}$ are linearly independent. Since $x_{0}=0$, the result follows.
\end{proof}

\section*{Acknowledgements}
The author would like to thank the following organisations for his Research Training Program scholarship: the Department of Education and Training, Australian Government and the School of Mathematics and Statistics, UNSW. The author would also like to thank Ian Doust for his help and supervision received whilst putting together this article.


\begin{thebibliography}{999}
	
	\bibitem
	{Doust 1} I. Doust, G. Robertson, A. Stoneham, A. Weston,
	`Distance matrices of subsets of the Hamming cube.'
	(Preprint).
	
	\bibitem
	{Enflo 1} P. Enflo,
	`On a problem of Smirnov.'
	\textit{Ark. Mat.}, \textbf{8} (1969), 107--109.
	
	\bibitem
	{Faver 1} T. Faver, K. Kochalski, M. Murugan, H. Verheggen, E. Wesson and A. Weston,
	`Roundness properties of ultrametric spaces.'
	\textit{Glasg. Math. J.}, \textbf{56} (2014), no. 3, 519--535.
	
	\bibitem
	{Gantmacher 1} F. R. Gantmacher,
	`The theory of matrices.'
	\textit{Vol. I, Chelsea, New York,} (1959), x + 374 pp.
	
	\bibitem{Hjorth 1} P. Hjorth, P. Lison\u{e}k, S. Markvorsen, C. Thomassen,
	`Finite metric spaces of strictly negative type.'
	\textit{Linear Algebra Appl.}, \textbf{270} (1998), 255--273.
	
	\bibitem{Lennard 1} C. J. Lennard, A. M. Tonge, A. Weston,
	`Generalized roundness and negative type.'
	\textit{Michigan Math. J.}, \textbf{44} (1997), 37--45.
	
	\bibitem{Maehara 1} H. Maehara,
	`Euclidean embeddings of finite metric spaces.'
	\textit{Discrete Math.}, \textbf{313} (2013), no. 23, 2848--2856.
	
	\bibitem{Murugan 1} M. K. Murugan,
	`Supremal $p$-negative type of vertex transitive graphs.'
	\textit{J. Math. Anal. Appl.}, \textbf{391} (2012), 376--381.
	
	\bibitem
	{Sanchez 1} S. S\'{a}nchez,
	`On the supremal $p$-negative type of finite metric spaces.'
	\textit{J. Math. Anal. Appl.}, \textbf{389} (2012), no. 1, 98--107.
	
	\bibitem
	{Schoenberg 1} I. J. Schoenberg,
	`On certain metric spaces arising from Euclidean spaces by a change of metric and their imbedding in Hilbert space.'
	\textit{Ann. of Math. (2)}, \textbf{38} (1937), no. 4, 787--793.
	
	\bibitem
	{Schoenberg 2} I. J. Schoenberg,
	`Metric spaces and positive definite functions.'
	\textit{Trans. Amer. Math. Soc.}, \textbf{44} (1938), no. 3, 522--536.
	
\end{thebibliography}
\end{document}